\documentclass[oneside,leqno,draft,10pt]{amsart}
\usepackage{latexsym,amssymb,amsmath,graphics,color}
\usepackage[dvips]{graphicx}
\usepackage{dsfont}

\makeatletter
\@namedef{subjclassname@2010}{%
  \textup{2010} Mathematics Subject Classification}
\makeatother

\newtheorem{thm}{Theorem}[section]
\newtheorem{cor}[thm]{Corollary}
\newtheorem{lem}[thm]{Lemma}
\newtheorem{prop}[thm]{Proposition}

\theoremstyle{definition}
\newtheorem{defin}[thm]{Definition}

\newtheorem{rem}[thm]{Remark}

\newtheorem*{rem*}{Remark}

\numberwithin{equation}{section}
\newcommand{\8}{\infty}
\newcommand{\eps}{\varepsilon}
\newcommand{\s}{\sigma}
\renewcommand{\a}{\alpha}

\usepackage{color}

\newcommand{\E}{\mathbb{E}}
\newcommand{\Prob}{\mathbb{P}}

\newcommand{\Pfs}[1][]{\ensuremath{\mathbb{P}_{#1}\text{-a.s.}}}
\newcommand{\Erw}[2][]{\ensuremath{\mathbb{E}_{#1} \left( {#2} \right)}}

\newcommand{\N}{\mathbb{N}}

\newcommand{\R}{\mathbb{R}}
\newcommand{\B}{\mathfrak{B}}

\newcommand{\V}{\mathbb{V}}

\newcommand{\F}{\mathcal{F}}

\newcommand{\llam}{{l}}

\newcommand{\supp}{\mathrm{supp}\,  }

\renewcommand{\epsilon}{\varepsilon}
\renewcommand{\rho}{\varrho}

\newcommand{\1}[1][]{\mathbf{1}_{#1}}
\newcommand{\norm}[1]{\ensuremath{\left\| {#1} \right\|}}

\newcommand{\abs}[1]{\ensuremath{\left| {#1} \right|}}
\newcommand{\skalar}[1]{\langle #1 \rangle}

\newcommand{\eqdist}{\stackrel{d}{=}}

\newcommand{\Step}[1][]{\textsc{Step #1}}

\newcommand{\ST}{\mathcal{S}}

\newcommand{\Pset}[1][]{\mathcal{P}_{#1}}

\newcommand{\law}[1]{\mathcal{L}\left( #1 \right)}
\newcommand{\tree}{\mathfrak{T}}

\newcommand{\Id}{\matrix{Id}}

\newcommand{\Rdnn}{{\R^d_\ge}}

\newcommand{\Rp}{\R_>}
\newcommand{\Rnn}{\R_\ge}

\newcommand{\Cf}[2][]{\mathcal{C}^{#1}\left( #2 \right)}

\renewcommand{\B}{\mathcal{B}}

\newcommand{\condC}{(C)}

\newcommand{\Sp}{\mathbb{S}_\ge}
\newcommand{\Sd}{S^{d-1}}
\newcommand{\Mset}{\mathcal{M}}
\newcommand{\interior}[1]{\mathrm{int}({#1})}
\newcommand{\est}[1][s]{e^*_{#1}}
\newcommand{\Pst}[1][s]{P_*^{#1}}
\newcommand{\nust}[1][s]{\nu^*_{#1}}
\newcommand{\pist}[1][s]{\pi_*^{#1}}

\newcommand{\Ps}[1][s]{P^{#1}}

\newcommand{\Qs}[1][s]{Q^{#1}}

\newcommand{\nus}[1][s]{\nu_{#1}}
\newcommand{\mM}{\mathbf{M}}
\newcommand{\mT}{\mathbf{T}}
\newcommand{\mPi}{\matrix{\Pi}}
\newcommand{\mL}{\matrix{L}}

\renewcommand{\matrix}[1]{\mathbf{#1}}

\newcommand{\mb}{\matrix{b}}

\newcommand{\LTfp}{\psi}

\newcommand{\as}{\cdot}

\newcommand{\LTa}[1][]{\phi^{#1}}

\newcommand{\T}{T}

\newcommand{\Lp}[2][p]{L^{#1}\left( #2 \right)}
\newcommand{\deins}{\mathbf{\vartheta_d}}
\newcommand{\eins}{\mathbf{\vartheta_1}}

\newcommand{\ma}{\matrix{a}}

\newcommand{\sline}[1][]{\mathcal{I}_{#1}}
\newcommand{\slineu}[1][]{\mathcal{I}_{#1}^u}

\renewcommand{\P}[2][]{\ensuremath{\mathbb{P}_{#1} \left( {#2} \right)}}
\newcommand{\Ex}[2][]{\ensuremath{\mathbb{E}_{#1}^\alpha \left( {#2} \right)}}

\newcommand{\mao}{{\ma_0}}

\renewcommand{\est}[1][s]{H^{#1}}
\renewcommand{\1}[1][]{\mathds{1}_{#1}}
\renewcommand{\deins}{\mathbf{1}}

\begin{document}





\title{Fixed Points of the Multivariate Smoothing Transform: The Critical Case}

    \author{Konrad Kolesko$^*$, Sebastian Mentemeier$^\dagger$}
 \address{$^*$ {Uniwersytet Warszawski\\Instytut Matematyki\\ul. Banacha 2\\
02-097 Warszawa, Poland} \\ $^\dagger$ Uniwersytet Wroc\l awski \\ Instytut Matematyczny \\ pl. Grunwaldzki 2/4 \\ 50-384 Wroc\l aw, Poland}
 \email{kolesko@math.uni.wroc.pl, mente@math.uni.wroc.pl}


\begin{abstract}
Given a sequence $(\mT_1, \mT_2, \dots)$ of random $d \times d$ matrices with nonnegative entries, suppose there is a random vector $X$ with nonnegative entries, such that 
$ \sum_{i \ge 1} \mT_i X_i $ has the same law as $X$, where $(X_1, X_2, \dots)$ are i.i.d.~copies of $X$, independent of $(\mT_1, \mT_2, \dots)$. 
Then (the law of) $X$ is called a fixed point of the multivariate smoothing transform. Similar to the well-studied one-dimensional case $d=1$, a function $m$ is introduced, such that the existence of $\alpha \in (0,1]$ with $m(\alpha)=1$ and $m'(\alpha) \le 0$ guarantees the existence of nontrivial fixed points. We prove the uniqueness of fixed points in the critical case $m'(\alpha)=0$ and describe their tail behavior. This complements recent results for the non-critical multivariate case. Moreover, we introduce the multivariate  analogue of the derivative martingale and prove its convergence to a non-trivial limit.
\end{abstract}



\maketitle


\section{Introduction}

Let $d \ge 2$ and
 $(\mT_i)_{i \ge 1}$ be a sequence of random $d\times d$-matrices  with nonnegative entries.  Assume that $$N := \# \{ i \, : \, \mT_i \neq 0 \}$$ is finite a.s. We will presuppose throughout that the $(\mT_i)_{i \ge 1}$ are ordered in such a way that $\mT_i \neq 0$ if and only if $i \le N$.  Given a random variable $X \in \Rdnn=[0,\infty)^d$, let $(X_i)_{i \ge 1}$ be i.i.d.~copies of $X$ and independent of $(\mT_i)_{i \ge 1}$. Then
$ \sum_{i=1}^N \mT_i X_i$
defines a new random variable in $\Rdnn$. If it holds that 
\begin{equation} \label{eq:FPE} X \eqdist \sum_{i=1}^N \mT_i X_i, \end{equation}
where $\eqdist$ means same law, then we call the law $\law{X}$ of $X$ a fixed point of the multivariate smoothing transform (associated with $(\mT_i)_{i \ge 1}$). By an slight abuse of notation, we will also call $X$ a fixed point.

This notion goes back to Durrett and Liggett \cite{DL1983}. For $d=1$, they proved (see also \cite{Liu1998,ABM2012}) that properties of fixed points are encoded in the function $m(s) := \E \sum_{i=1}^N T_i^s$ (here $(T_i)_{i\ge 1}$ are nonnegative random numbers): If $m(\alpha) =1$ and $m'(\alpha) \le 0$ for some $\alpha \in (0,1]$ and some non-lattice and moment assumptions are satisfied, then there is a fixed point which is unique up to scaling. Conversely, the condition $m(\alpha) =1$, $m'(\alpha) \le 0$ for some $\alpha \in (0,1]$ is also necessary for the existence of fixed points. 

Moreover, if $\LTfp(r) = \Erw{e^{-rX}}$ is the Laplace transform of a fixed point, then there is a positive function $L$, slowly varying at $0$, and $K >0$ such that
\begin{equation} \label{eq:slowvar1} \lim_{r \to 0} \frac{1-\LTfp(r)}{L(r) r^\alpha} = K.  \end{equation}
The function $L$ is constant if $m'(\alpha) < 0$ and $L(t) = (\abs{ \log t} \vee 1)$ if $m'(\alpha)=0$, the latter being called the \emph{critical case}. For $\alpha<1$, the property \eqref{eq:slowvar1} implies that the fixed points have Pareto-like tails with index $\alpha$, i.e.~$ \lim_{t \to \infty} t^{-\alpha}\P{X>t}/L(1/t) \in (0, \infty)$, see \cite{Liu1998} for details. Tail behavior in the case $\alpha=1$, in which there is no such implication, is investigated in  \cite{Guivarch1990,Liu1998,Buraczewski2009}. 

Existence and uniqueness results in the multivariate setting $d \ge 2$ for the non-critical case have been recently proved in \cite{Mentemeier2013}. The aim of this note is to provide the corresponding result for the multivariate critical case. In order to so, we will first review necessary notation and definitions from \cite{Mentemeier2013}, in particular introducing the multivariate analogue of the function $m$, as well as a result about the existence of fixed points in the critical case. Following the approach in \cite{Biggins1997,Biggins2005b,Kyprianou1998} we will then prove that a multivariate regular variation property similar to \eqref{eq:slowvar1} holds for fixed points (with an essentially unique, but yet undetermined slowly varying function $L$), which we use in order to prove the uniqueness of fixed points, up to scalars. Under some extra (density) assumption, we identify the slowly varying function to be the logarithm also in the multivariate case, which allows us to introduce and prove convergence of the multivariate version of the so-called \emph{derivative martingale}, a notion coined in \cite{Biggins2004}. It appears prominently in the limiting distribution of the minimal position in branching random walk, see \cite{Aidekon2013,Aidekon2014,Biggins2004} for details and further references.


\section{Statement of Results}

We start by introducing the assumptions and some notation needed therefore.
Write $\Pset(\Rdnn)$ for the set of probability measures on $\Rdnn$ and  $\Mset:=M(d\times d,\Rnn)$ for the set of $d \times d$-matrices with nonnegative entries. Given a sequence $\T:=(\mT_i)_{i \ge 1}$ of random matrices  from $\Mset$, only the first $N$ of which are nonzero, with $N< \infty$ a.s., we aim to determine the set of fixed points of the mapping $\ST : \Pset(\Rdnn) \to \Pset(\Rdnn),$
$$ \ST \eta := \law{\sum_{i=1}^N \mT_i X_i}, \qquad \text{for $(X_i)_{i\ge 1}$ i.i.d.~ with law $\eta$ and independent of $(\mT_i)_{i \ge 1}$}.$$
Without further mention, we assume $(\Omega, \B, \Prob)$ to be a probability space which is rich enough to carry all the occurring random variables.

\subsection{The weighted branching process and iterations of $\ST$}
Let $\V := \bigcup_{n=0}^\infty \N^n$ be a tree with root $\emptyset$ and Ulam-Harris labeling. We write $\abs{v}=n$ if $v=v_1 \cdots v_n \in \{1, \dots,N\}^n$, $v|k = v_1 \cdots v_k$ for the ancestor in the $k$-th generation and $vi=v_1 \cdots v_n i$ for the $i$-th child of $v$, $i \in \N$. 

To each node $v \in \V$ assign an independent copy $\T(v)$ of $\T$ and, given a random variable $X \in \Rdnn$, as well an independent copy $X(v)$ of $X$, such that
$(\T(v))_{v \in \V}$ and $(X(v))_{v \in \V}$ are independent. Introduce a filtration by
$$ \B_n ~:=~ \sigma\bigg( (\T(v))_{\abs{v}<n}\bigg).$$
Upon defining  recursively the product of weights along the path from $\emptyset$ to $v$ by
$$ \mL(\emptyset) := \Id, \qquad \mL(vi) = \mL(v) \mT_i(v),  $$
we obtain the iteration formula
$$ \ST^n \law{X} = \law{\sum_{\abs{v}=n} \mL(v) X(v)}, $$
which in terms of Laplace transforms $\LTa(x) = \E \bigg[e^{- \skalar{x,X}} \bigg]$ becomes
\begin{equation}\label{eq:STLT} \ST^n \LTa(x) = \E \bigg[ \prod_{\abs{v}=n} \LTa( \mL(v)^\top x) \bigg], \qquad x \in \Rdnn. \end{equation}

\subsection{Assumptions}
As noted before, we assume
\begin{equation}\tag{A1}\label{A1} \text{the r.v. $N := \# \{ i \, : \, \mT_i \neq 0 \}$ equals $\sup\{ i \, : \, \mT_i \neq 0 \} $ and is finite a.s.} \end{equation}
and
\begin{equation*}
N \ge 1 \text{ a.s. and } 1 < \E N < \infty.
\end{equation*}
 This assumption guarantees, that the underlying Galton-Watson tree (consisting of the nodes $v$ with $\mL(v) \neq 0$) is supercritical and allows to define a probability measure $\mu$ on $\Mset$ by
\begin{equation}\label{eq:defmu}
\int \, f(\ma) \, \mu(d\ma) ~:=~ \frac{1}{\E N}\Erw{\sum_{i=1}^N f(\mT_i) }.
\end{equation}
On the (support of the) measure $\mu$, we will impose the following condition $\condC$: 
\begin{defin}\label{defn:condc}
A subsemigroup $\Gamma \subset \Mset$ satisfies condition $\condC$, if 
\begin{enumerate}
\item every $\ma$ in $\Gamma$ is {\em allowable}, i.e. it has no zero row nor column, and
\item $\Gamma$ contains a matrix with all entries  positive $(>0)$.
\end{enumerate} 
\end{defin}
For the measure $\mu$ as defined in Eq. \eqref{eq:defmu}, we assume
\begin{equation}
\tag{A2} \label{A3} \text{ The subsemigroup $[\supp \mu]$ generated by $\supp \mu$ satisfies $\condC$.}
\end{equation}
Note that if $\ma \in \Mset$ is an allowable matrix, then we can define its action on $\Sp := \Sd \cap \Rdnn$ by
$$ \ma \as u ~:=~ \frac{\ma u}{\abs{\ma u}}, \qquad u \in \Sp.$$
Furthermore, we need a multivariate analogue of a non-lattice condition: Recall that a matrix $\ma$ with all entries positive has a algebraic simple dominant eigenvalue $\lambda_\ma >0$ with corresponding normalized eigenvector $v_\ma$ the entries of which are all positive.
\begin{equation}
\tag{A3} \label{A4} \text{The additive group generated by $\{ \log \lambda_{\ma} \, : \, \ma \in [\supp \, \mu] $ has all entries positive$\}$ is dense in $\R$}
\end{equation} 

Let $\mM, (\mM_n)_{n \in \N}$ be i.i.d. random matrices with law $\mu$, and write $\mPi_n := \prod_{i=1}^n \mM_n$. Then it is shown in \cite{Mentemeier2013}, that the multivariate analogue of the function $m$ is given by
$$ m(s) := \E[N] \lim_{n \to \infty} \left( \E \norm{\mPi_n}^s \right)^{1/n},$$
which is finite on
$$ I_\mu := \{ s > 0 \, : \, \E\bigg[ \norm{\mM}^s  \bigg]< \infty \}.$$
On $I_\mu$, it is log-convex, and thus the left-handed derivatives $m'(s^-)$ exist.

We assume to be in the critical case, i.e. 
\begin{equation}
\label{A5}\tag{A4} \text{there is $\alpha \in (0,1] \cap I_{\mu}$ with $m(\alpha)=1$ and $m'(\alpha^-)=0$.} 
\end{equation}

For the multivariate case, the classical \emph{T-log T} condition splits into an upper bound
and a lower bound: Introducing $\iota(\ma) := \inf_{u \in \Sp}\abs{\ma u}$, we observe that $\iota(\ma) >0$ for $\ma \in \Mset$, and that for all $u \in \Sp$,
$$ \iota(\ma) ~ \le ~ \abs{\ma u} ~\le ~ \norm{\ma}.$$ Note that if $\ma$ is invertible, then $\norm{\ma^{-1}}^{-1} \le \iota(\ma)$.
\begin{equation}
 \tag{A5}\label{A6}
 \E \bigg[ \norm{\mM}^\alpha \log(1+ \norm{\mM}) \bigg]< \infty, \qquad \E \bigg[ (1+\norm{\mM})^\alpha \abs{\log \iota(\mM^\top)} \bigg] < \infty \\
\end{equation}
Sometimes we will impose the stronger condition
\begin{equation} \label{A7} \tag{A6}
\text{There is $c >0$ such that $\P{\iota(\mM^\top) \ge c}=1$,}
\end{equation}
which together with the first part of \eqref{A6} implies the second part of \eqref{A6}.

\medskip

In the second part of the paper, we will need stronger assumptions on $\mu$, which guarantee that the associated Markov random walk (to be defined below) is Harris recurrent. We will consider the absolute continuity assumption
\begin{equation}\label{A4c}\tag{A3c}
\exists\, \mao \in \interior{\Mset} \, \exists\, \gamma_0, c >0 \text{ s.t. } \P{\mM \in \cdot  } ~\ge~ \gamma_0 \, \llam^{d \times d}(\cdot \cap B_c(\mao)),
\end{equation}
where $\llam^{d\times d}$ denotes the Lebesgue measure on the set of $d \times d$-matrices, seen as a subset of $\R^{d^2}$. A similar assumption for invertible matrices appears in \cite[Theorem 6]{Kesten1973} and subsequently in \cite{AM2010}. It is easy to check that \eqref{A4c} implies \eqref{A4}.

We will consider as well a quite degenerate case, namely
\begin{equation}
\label{A4f}\tag{A3f}  \text{$\supp \mu$ is finite and consists of rank-one matrices, and \eqref{A4} holds.}
\end{equation} 
Note that an allowable rank-one matrix $\ma$ has all entries positive, the columns are multiples of a vector $v_\ma \in \interior{\Sp}$, and consequently,  $\ma \as u = v_\ma$ for all $u \in \Sp$.

We will also impose a stronger moment condition, namely
\begin{equation}\label{A8}\tag{A7}
\E\bigg[ N^{p_0}+{\left(\sum_{i=1}^N\norm{\mT_i}\right)}^{p_1}\bigg]<\8 \qquad\text{for some } p_0,p_1\ge1\text{ such that }p_0+p_1>2.
\end{equation}

Note that \eqref{A8} implies  
\begin{equation}
 \label{eq:moments_assumption}
 \E\bigg[\left(\sum_{i=1}^N\norm{\mT_i}^{\alpha/(1+\delta)}\right)^{1+\delta}\bigg]<\8 \text { for small enough }\delta>0. 
\end{equation}
Indeed, for any $0<s<1$ and $p$ such that $1/p=(1-s)/p_0+s/p_1$, using first Jensen's and then H\"{o}lder's inequality, the random variable $\sum_i^N\norm{\mT_i}^s$ has  finite  moment of the order $p$:
\begin{align*}
 \E\bigg[\left(\sum_i^N\norm{\mT_i}^s\right)^{p}\bigg]\le\E\bigg[N^{(1-s)p} \left(\sum_i^N\norm{\mT_i}\right)^{sp}\bigg]
 \le(\E N^{p_0})^{\frac{(1-s)p}{p_0}}\left(\E\bigg[\bigg(\sum_i^N\norm{\mT_i}\bigg)^{p_1}\bigg]\right)^{\frac{sp}{p_1}}<\8.
\end{align*}

\subsection{Previous Results}

We have the following existence result in the critical case.

\begin{prop} \label{prop:existence}
Assume \eqref{A1} --  \eqref{A3} and \eqref{A5} -- \eqref{A6}. Then Eq. \eqref{eq:FPE} has a nontrivial fixed point. 
\end{prop}

\begin{proof}[Source:] 
Theorem 1.2 in \cite{Mentemeier2013}.
\end{proof}

The main contribution of this paper is to prove the uniqueness of this fixed point, and to give asymptotic properties of its Laplace transform. It is convenient to introduce polar coordinates $(r,u) \in [0, \infty) \times \Sp$ on $\Rdnn$. Moreover, we will use that for $s \in I_{\mu}$, the operators $\Ps$ and $\Pst$, being self-mappings of the set $\Cf{\Sp}$ of continuous functions on $\Sp$ and defined by
$$ \Ps f(u) := \Erw{\abs{\mM u}^s \, f(\mM \as u)}, \qquad \Pst f(u) := \Erw{\abs{\mM^\top u}^s f(\mM^\top \as u)},$$
are quasi-compact with spectral radius equal to $k(s):=(\E\, N)^{-1} m(s)$ and there is a unique positive continuous functions $\est \in \Cf{\Sp}$ and unique probability measures $\nus, \nust \in \Pset(\Sp)$ such that
\begin{align}\label{eq:eigenfunctions}
\Pst \est = \frac{m(s)}{\E N} \est, \qquad \nus \Ps = \frac{m(s)}{\E N} \nus, \qquad \nust \Pst = \frac{m(s)}{\E N} \nust
\end{align}
and the following relation holds:
\begin{equation}
\label{eq:estnus} \est(u) ~=~ \int_{\Sp}  \skalar{u,y}^s \, \nus(dy) \qquad \text{ for all $u \in \Sp$.}
\end{equation}
See \cite{BDGM2014} for details and proofs. Using Eq. \eqref{eq:estnus}, we can extend $\est$ to a $s$-homogeneous function on $\Rdnn$, i.e. 
$$ \est(x) ~:=~ \int_{\Sp}  \skalar{x,y}^s \, \nus(dy) ~=~ \abs{x}^s H^s\bigg( \frac{x}{\abs{x}} \bigg), \qquad x \in \Rdnn. $$
Using $\est$, we are now going to provide a many-to-one lemma.

Let $u \in \Sp$. Define for $v \in \V$
$$ S^u(v) := -\log \abs{\mL(v)^\top u}, \qquad U^u(v) := \mL(v)^\top \as u.$$ Then, by Eq. \eqref{eq:eigenfunctions},
we see that 
\begin{align} \overline{\Qs}f(u,t) ~:=&~ \frac{1}{\est(u)\, m(s)} \E \left[ \sum_{i=1}^N \, f(U^u(i), t-S^u(i)) e^{-s S^u(i)} \est(U^u(i)) \right] \label{def:Q}\\
=&~ \frac{1}{\est(u)\, m(s)} \E N \, \E \Bigl[{f(\mM^\top \as u, t+ \log \abs{\mM^\top u}) \, \abs{\mM^\top u}^\alpha \est(\mM^\top \as u)}\Bigr] \nonumber
\end{align}
defines a Markov transition operator on $\Sp \times \R$. Let $(U_n, S_n)_n$ be a Markov chain in $\Sp \times \R$ with transition operator $\overline{\Qs[\alpha]}$ and denote the probability measure on the path space $(\Sp \times \R)^\N$ with initial values $(U_0,S_0)=(u,s)$ by $\Prob_{u,s}^\alpha$ and the corresponding expectation symbol by $\E_{u,s}^\alpha$. Most times, we will use the shorthand notations $\Prob_u^\alpha = \Prob_{u,0}^\alpha$ and $\Prob_{\eta}^\alpha = \int \, \Prob_{u,s}^\alpha \, \eta(du,ds)$ for a probability measure $\eta$ on $\Sp \times \R$.  

\begin{prop}\label{prop:many to one} 
For all $s \in I_\mu$, $u \in \Sp$, $n \in \N$ and measurable $f : (\Sp \times \R)^{n+1} \to \R$,
\begin{equation*}\label{eq:many to one}
\frac{1}{\est[s](u)\, m(s)^n} \, \E \left[ \sum_{\abs{v}=n} \, f\Bigl( (U^u\bigl(v|k \bigr), S^u\bigl(v|k \bigr))_{k \le n}\Bigr) \, e^{-s S^u(v)} \est(U^u(v)) \right]
=~ \E_u^s  f(U_0, S_0, \cdots, U_n, S_n).
\end{equation*}
\end{prop}

\begin{proof}[Source:] Corollary 4.3 in \cite{Mentemeier2013}. \end{proof}

We call $(U_n, S_n)_{n \in \N}$ 
the {\em associated Markov random walk}. It generalizes the concept of the associated random walk in \cite{DL1983,Liu1998}. In particular, it holds for all $u \in \Sp$, that
$$ \lim_{n \to \infty} \frac{S_n}{n} ~=~ 0 \qquad \Prob_u^\alpha\text{-a.s.}, $$
see \cite[Theorem 6.1]{BDGM2014}. Moreover, it is shown in \cite[Lemma 7.1]{Buraczewski2014} that
$$ b(u):= \lim_{n \to \infty} \E_u^\alpha S_n$$
is well defined and continuous, and satisfies
\begin{equation}\label{eq:b} \E_u^\alpha [S_1 { + } b(U_1)] = b(u) . \end{equation}

Using Eq. \eqref{eq:b}, we obtain that
$$ \mathcal{W}_n(u) := \sum_{\abs{v}=n} \left[ S^u(v) { + } b(U^u(v)) \right] \, \est[\alpha](U(v)) e^{-\alpha S(v)} $$
defines a martingale with respect to the filtration $\B_n$, which we will show to be the multivariate analogue of the derivative martingale. In fact, $b$ can be considered as the derivative of $\est[\alpha]$, see \cite[(7.9)]{Buraczewski2014}.

\subsection{Main Results} 
Our first result proves that, upon imposing the non-lattice condition \eqref{A3} and the stronger moment reap. boundedness assumptions \eqref{A7}--\eqref{A8}, the fixed point given by Proposition \ref{prop:existence} is unique up to scaling, and satisfies an multivariate analogue of the regular variation property \eqref{eq:slowvar1}.

\begin{thm}\label{thm:main}
Assume \eqref{A1} -- \eqref{A8} . Then there is a random measurable function $Z : \Sp  \to [0, \infty)$ with $\P{Z(u)>0}=1$ for all $u \in \Sp$, such that $X$  is a nontrivial fixed point of \eqref{eq:FPE}
on $\Rdnn$ if and only if its Laplace transform satisfies
\begin{equation}\label{LTofFP} \LTfp(ru) ~:=~ \Erw{e^{-r \skalar{u,X}}} ~=~ \Erw{e^{-r^\alpha K Z(u)}} \qquad \forall u \in \Sp,\, r \in \Rnn \end{equation}
for some $K > 0$.

There is an essentially unique positive function $L$, slowly varying at $0$ with $\liminf_{r \to 0} L(r) = \infty$, such that 
\begin{equation} \label{eq:regvar}\lim_{r \to 0} \frac{1-\LTfp(ru)}{L(r) \, r^\alpha} ~=~ K \est[\alpha](u). \end{equation}
\end{thm}

\begin{rem*} Essentially unique means that if $L_1$ and $L_2$ satisfy Eq. \eqref{eq:regvar}, then $\lim_{r \to 0} L_1(r)/L_2(r)=1$.  Depending on the value of $\alpha$, additional information can be extracted from Eq. \eqref{eq:regvar}.
\begin{enumerate}
\item If $\alpha <1$, then a Tauberian theorem (see \cite[XIII.(5.22)]{Feller1971}) together with \cite[Theorem 1.1]{BDM2002} implies the following multivariate regular variation property
$$ \lim_{r \to \infty} \frac{\P{\abs{X} > sr, \ \frac{X}{\abs{X}} \in \cdot}}{\P{\abs{X} > r}} = s^{-\alpha} \nus[\alpha],$$
see \cite[Section 6]{Mentemeier2013} for details.
\item If $\alpha=1$, then $\E \abs{X}=\infty$ for every non-trivial fixed point, see Lemma \ref{lem:Linfty}. Moreover, the aperiodicity condition \eqref{A4} is not needed, see Remark \ref{rem:aperiodic}. This is in analogy with the one-dimensional situation, see e.g.~\cite[Corollary 1.5]{Liu1998}.
\end{enumerate}
\end{rem*}

Upon imposing the additional assumptions \eqref{A4c} or \eqref{A4f} on $\mu$, we will identify the function $L$ as well as the random variable $Z$.

\begin{thm}\label{thm:main2}
Assume \eqref{A1} --  { \eqref{A8}}, with \eqref{A4c}  or \eqref{A4f} instead of \eqref{A4}. 
Then $\mathcal{W}_n(u)$ converges a.s. to a nonnegative limit $\mathcal{W}(u)$ with $\P{\mathcal{W}(u)>0}=1$, and a random variable $X \in \Rdnn$ is a nontrivial fixed point of \eqref{eq:FPE} if and only if for some $K>0$,
$$\Erw{e^{-r \skalar{u,X}}} ~=~ \Erw{e^{-r^\alpha K \mathcal{W}(u)}}  \qquad \forall u \in \Sp,\, r \in \Rnn.$$
Moreover, the slowly varying function $L$ in Eq. \eqref{eq:regvar} can be chosen as (a scalar multiple of) $L(r) = \abs{\log r} \vee 1$.
\end{thm}

\subsection{Structure of the Paper}

The further organization is as follows: In Section \ref{sect:preliminaries}, we study the associated Markov random walk, which is recurrent due to the criticality assumption. Under assumptions \eqref{A4c}, a regeneration property known from the theory of Harris recurrent Markov chains will be shown to hold.  In Section \ref{sect:regular variation}, we prove that each fixed point satisfies \eqref{eq:regvar}, which is a main ingredient in the proof of uniqueness in Section \ref{sect:uniqueness}. 
In Section \ref{sect:L}, we turn to the proof of Theorem \ref{thm:main2} and study the behavior of the Laplace transform of the fixed point. We conclude with Section \ref{sect:derivative martingale}, where the convergence of the derivative martingale is proved.

\subsection*{Acknowledgements}
The main part of this work was done during mutual visits to the Universities of Muenster and Warsaw, to which we are grateful for hospitality. S.M.~was partlially supported by the Deutsche Forschungsgemeinschaft (SFB 878). K.K.~was partially supported by NCN grant DEC-2012/05/B/ST1/00692. 

\section{The Associated Markov Random Walk}\label{sect:preliminaries}

In this section, we provide additional information about the associated Markov random walk, in particular about its stationary distribution and recurrence properties. Moreover, we show that it is Harris recurrent and satisfies a minorization condition under the additional assumption \eqref{A4c}.

\subsection{The Associated Markov Random Walk}

The Markov chain $(U_n, S_n)_n$ constitutes a Markov random walk, i.e. for each $n \in \N$, the increment $S_n - S_{n-1}$ depends on the past only through $U_{n-1}$, this follows from the definition of $\overline{\Qs[\alpha]}$. 
Such Markov random walks which are generated by the action of nonegative matrices where first studied by Kesten in his seminal paper \cite{Kesten1973}, and very detailed results are given in \cite{BDGM2014}. For the reader's convenience, we cite those who are important for what follows. Recall that we denoted the Perron-Frobenius eigenvalue and the corresponding normalized eigenvector of a matrix $\ma \in \interior{\Mset}$ by $\lambda_\ma$ resp. $v_\ma$.

\begin{prop}\label{prop:UnSn}
Assume \eqref{A1} -- \eqref{A3} and let $\alpha \in I_\mu$ ($m(\alpha)=1$ is not needed here). For this $\alpha$, assume \eqref{A6}.
Then the following holds:
\begin{enumerate}
\item The Markov chain $(U_n)_n$ on $\Sp$ has a unique stationary distribution $\pist[\alpha]$ under $\Prob_u^\alpha$, with density (proportional to) $\est[\alpha]$ w.r.t~the measure $\nust[\alpha]$. \label{Un ergodic}
\item $\supp \pist[\alpha] = \overline{\{ v_\ma \, : \, \ma \in [\supp \mu] \cap \interior{\Mset} \}}.$
\item For all $u \in \Sp$, \label{SLLN}
$$ \lim_{n \to \infty} \frac{S_n}{n} =  \E_{\pist[\alpha]}^\alpha S_1 = \int_{\Sp}\, \E_u^\alpha S_1 \, \pist[\alpha](du) = \frac{m'(\alpha^{-})}{m(\alpha)} \qquad \Prob_u^\alpha\text{-a.s.}$$
\end{enumerate}
Now assume $1 \in I_\mu$ and that \eqref{A6} holds for $\alpha=1$. Then $\mb :=\E \mM \in \interior{\Mset}$.
\begin{enumerate}
\item[(4)] $m(1) = (\E N) \lambda_{\mb}$ and $\est[1](u) = \skalar{u,v_{\mb}}$.
\item[(5)]  The derivative of $m$ at 1 can be calculated to $$m'(1^{-})= \int_{\Sp} \,  \Erw{\frac{\skalar{\mM u,v_{\mb}}}{\skalar{u,v_{\mb}}} \, \log \abs{\mM u} } \, \pist[1](du)$$ 
\end{enumerate}
\end{prop}

\begin{proof}[Source: ] Sections 4 and 6 of \cite{BDGM2014}.\end{proof}

%

\subsection{Recurrence of Markov Random Walks}

By Proposition \ref{prop:UnSn} \eqref{SLLN}, in the critical case $m'(\alpha^{-})=0$ the Markov random walk $(S_n)_n$ is centered in the stationary regime and satisfies a strong law of large numbers. Alsmeyer \cite{Alsmeyer2001} studied recurrence properties of such Markov random walks, which we will make use of. 
%
%

\begin{lem}\label{lem:recurrence}
Assume that \eqref{A1}-\eqref{A6} hold. For any open set $A$ with $\pist[\alpha](A) > 0$ and any open interval $B \subset \R$, it holds that 
\begin{equation}\label{recurrence}\Prob_{\pist[\alpha]}^\alpha ({(U_n, S_n) \in A \times B \text{ infinitely often}})=1. \end{equation}
If the aperiodicity condition \eqref{A4} is not assumed, then still
\begin{equation} \label{eq:oscillates} \liminf_{n\to \infty}{S_n} ~=~ - \infty, \qquad \limsup_{n \to \infty}{S_n} ~=~ \infty \qquad \Prob_{\pist[\alpha]}^\alpha\text{-a.s.} \end{equation}
\end{lem}

\begin{proof}
Let $A$ be any open set $A$ with $\pist[\alpha](A) > 0$. By the strong law of large numbers for Markov chains (see \cite{Breiman1960}), 
\begin{equation}\label{SLLNmarkov} \lim_{n \to \infty} \frac{1}n \sum_{k=1}^n f(U_k) ~=~ \int\, f(x) \, \pist[\alpha](dx) \qquad \Prob_{\pist[\alpha]}^\alpha\text{-a.s.}, \end{equation}
thus, using $f=\1[A]$, we obtain that $\Prob_{\pist[\alpha]}^\alpha(U_n \in A \text{ infinitely often})=1 $.
Denote the successive hitting times of $A$ by $\tau_n$. Then $(U_{\tau_n}, S_{\tau_n})$ is again a Markov random walk, and $\pi_A :=\pist[\alpha](\cdot \cap A)/\pist[\alpha](A)$ is the stationary probability measure for $U_{\tau_n}$. The aperiodicity assumption \eqref{A4}  implies that $(U_n,S_n)$ are nonarithmetic in the sense of \cite{Alsmeyer2001}, see \cite{Buraczewski2014} for details. Lemma 1 in  \cite{Alsmeyer2001} gives that $(U_{\tau_n}, S_{\tau_n})$ is nonarithmetic as well. 
Using \eqref{SLLNmarkov} with $f=\1[A]$ again, this gives that $n/\tau_n \to \pi(A)$ a.s. Combining this with the strong law of large numbers  \eqref{SLLN} in Proposition \ref{prop:UnSn}, 
we deduce that
$$ \lim_{n \to \infty} \frac{S_{\tau_n}}{n} =  \lim_{n \to \infty} \frac{S_{\tau_n}}{\tau_n} \, \frac{\tau_n}{n} =  \frac{1}{\pist[\alpha](A)} \cdot 0 \qquad \Prob_{\pist[\alpha]}^\alpha\text{-a.s.}.$$

Then Theorem 2 in  \cite{Alsmeyer2001} (for the nonarithmetic case) gives that the recurrence set 
$$ \{ s \in \R \, : \, \text{ for all $\epsilon > 0$, }S_{\tau_n} \in (s-\epsilon, s+\epsilon) \text{ infinitely often } \}$$
 equal to $\R$, which shows that $\Prob_{\pist[\alpha]}^\alpha({S_{\tau_n} \in B \text{ infinitely often}})=1$.  
 
In the arithmetic case, the recurrence set is still a closed subgroup of $\R$, which implies the oscillation property.
\end{proof}

\begin{cor}\label{cor:uo}
There is $u_0 \in \interior{\Sp} \cap (\supp \pist[\alpha])$ such that
\begin{equation}\label{recurrence2}\Prob_{u_0}^\alpha ({(U_n, S_n) \in A \times B \text{ infinitely often}})=1, \text{ and } \end{equation}
\begin{equation} \label{eq:oscillates2} \liminf_{n\to \infty}{S_n} ~=~ - \infty, \qquad \limsup_{n \to \infty}{S_n} ~=~ \infty \qquad \Prob_{u_0}^\alpha\text{-a.s.} \end{equation}
\end{cor}

\begin{proof}
By Proposition \ref{prop:UnSn}, $\supp \pist[\alpha]$ consists of the (closure of the) set of normalized Perron-Frobenius eigenvectors of matrices $\ma \in [\supp \mu]$ with all entries strictly positive. By part (2) of $\condC$, this set is nonempty, hence $\interior{\Sp} \cap (\supp \pist[\alpha]) \neq \emptyset$ and even $\pist[\alpha](\interior{\Sp})=1$. On the other hand, Lemma \ref{lem:recurrence} implies validity of \eqref{recurrence2} and \eqref{eq:oscillates2} for $\pist[\alpha]$-a.e. $u \in \Sp$, hence we can find $u_0 \in \interior{\Sp}$ satisfying the assertions.
\end{proof}
 
%

\subsection{Implications of Assumptions \eqref{A4c} and \eqref{A4f}}

In this subsection, we explain how Assumptions \eqref{A4c} and \eqref{A4f} imply that the Markov chain $(U_n)_{n \in \N}$ has an atom (possibly after redefining it on an  extended probability space), which can be used to obtain a sequence $(\sigma_n)_{n \in \N}$ of {\em regeneration times} for the Markov random walk $(U_n, S_n)$, i.e.~stopping times such that$(U_{\sigma_n}, S_{\sigma_n}-S_{\sigma_{n-1}})_{n \in \N}$ becomes an i.i.d.~sequence. Namely, we are going to prove the following lemma for the Markov chain $(U_n, Y_n):=(U_n, S_{n}-S_{n-1})$. 

\begin{lem}\label{regenerationlemma}
Assume \eqref{A1}--\eqref{A3} and \eqref{A4c} or \eqref{A4f}. On a possibly enlarged probability space, one can redefine $(U_n, Y_n)_{n \geq 0}$ together with an increasing sequence $(\sigma_n)_{n \geq 0}$ of random times such that the following conditions are fulfilled under any $\Prob_u^\alpha$, $u\in \Sp$:
\begin{itemize}
\item[(R1)] There is a filtration $\mathcal{G}=(\mathcal{G}_n)_{n \geq 0}$ such that $(U_n, Y_n)_{n \geq 0}$ is Markov adapted and each $\sigma_n$ a stopping time with respect to $\mathcal{G}$, moreover, $\{\sigma_n = k\} \in \mathcal{G}_{k-1}$ for all $n,k \ge 0$. \label{R1}
\item[(R2)] The sequence $(\sigma_{n+1}-\sigma_{n})_{n\ge 1}$ is i.i.d.~ with law $\P[\eta]{{\sigma_1} \in \cdot}$ and is independent of $\sigma_{1}$.  \label{R2}
\item[(R3)] For each $k\ge 1$, $(U_{\sigma_k+n},Y_{\sigma_k+n})_{n \geq 0}$ is independent of $(U_j, Y_j)_{0 \leq j \leq \sigma_k -1}$ with  distribution $\Prob^\alpha_\eta((U_n, Y_n)_{n \geq 0} \in \cdot)$.  \label{R3} 
\item[(R4)] There is $q \in (0,1)$ and $l \in \N$ such that  $\sup_{u \in \Sp} \Prob_u^\alpha(\sigma_1 > ln) \le q^n$.
\end{itemize}
\end{lem}

This lemma is quite immediate under condition \eqref{A4f}, for Proposition \ref{prop:UnSn}, (2) shows that the unique stationary measure $\pist[\alpha]$ for $(U_n)$ under $\Prob_u^\alpha$ is supported on the {\em finite} set $\mathds{S} :=\{ v_\ma \, : \, \ma \in \supp \mu\}$ (note that $v_{\ma \mb} = v_{\ma}$ if $\ma$ has rank one, thus the semigroup $[\supp \mu]$ can be replaced by $\supp \mu$.) Moreover, independent of the initial value $u \in \Sp$, $U_1 \in \mathds{S}$ $\Prob_u^\alpha$-f.s., i.e. $\Sp \setminus \mathds{S}$ is uniformly transient for $(U_n)_{n \in \N}$, and thus we can study $(U_n)_{n \in \N}$ on the finite state space $\mathds{S}$. Then, if $(\sigma_n)_{n \in \N}$ is a sequence of successive hitting times of a point $u_0 \in \mathds{S}$, the assertions of the lemma follow from the theory of Markov chains with finite state space.

\medskip

\begin{rem}
A crucial point is that we also obtain the independence of $Y_{\sigma_k}$ from $(U_j, Y_j)_{0 \leq j \leq \sigma_k -1}$, thereby strengthening analogous results for invertible matrices, obtained in \cite{AM2010,Mentemeier2013a}.
\end{rem}

From now on, assume \eqref{A4c}. 
We are going to prove that the chain $(U_n, Y_n)$ satisfies a minorization condition as in \cite[Definition 2.2]{Athreya1978} resp. \cite[(M)]{Num1978}. 
If $v_\mao \in \Sp$ is the Perron-Frobenius eigenvalue of the matrix $\ma_0$ from \eqref{A4}, then we have the following result:
\begin{lem}\label{lem:min1}
For each $u \in \Sp$, $\delta >0$,
$$ \Prob_u^\alpha(U_n \in B_{\delta}(v_\mao) \text{ infintely often })=1,$$
moreover, if $\tau$ denotes the first hitting time of $B_\delta(v_\mao)$, then there is $l \ge 1$ and  $q_0 \in (0,1)$ such that 
$$ \sup_{u \in \Sp} \Prob_u^\alpha(\tau > ln) \le q_0^n, $$
i.e. $\tau/l$ is stochastically bounded by a random variable with geometric distribution.
\end{lem}

\begin{proof}[Source:]
This is proved in \cite[p.218-220, proof of I.1]{Kesten1973}, the crucial point being that $v_\mao$ is a strict contraction on $\Sp$ with attractive fixed point $v_\mao$, and small perturbations of $\mao$ still attract to a neighborhood of $v_\mao$, and such matrices are realized with positive probability.
\end{proof}

\begin{lem}\label{lem:min2}
There are $\delta >0$, $\gamma >0$ and a probability measure $\eta$ on $\mathcal{R}:=B_\delta(v_\mao) \times \R$  such that for all $u \in B_\delta(v_\mao)$ and all measurable subsets $A \subset B_\delta(v_\mao)$, $B \subset \R$ 
$$ \Prob_u^\alpha(U_1 \in A, Y_1 \in B) ~\ge~ \gamma \, \eta(A \times B).$$
\end{lem}

\begin{proof}
We follow the approach in \cite{AM2010, Mentemeier2013a}. 

\Step[1]. Given $c >0$, $\mao \in \interior{\Mset}$, there is $\epsilon >0$ such that for all orthogonal matrices $\matrix{O}$, satisfying $\norm{\matrix{O}-\Id}< \epsilon$, $B_{c/2}(\mao)\matrix{O} \subset B_c(\mao)$. {\em Proof}: Let $\mb \in B_{c/2}{\mao}$, then, since $\matrix{O}$ is an isometry, 
$$ \norm{\mb \matrix{O}- \mao} \le \norm{\mb\matrix{O} - \mao \matrix{O}} + \norm{\mao \matrix{O} - \mao} \le \norm{\mb - \mao} - \norm{\mao} \norm{\matrix{O} - \Id} \le c/2 + \epsilon \norm{\mao}.$$
\Step[2]. For all $\epsilon >0$ there is $\delta>0$ such that for each $u \in B_\delta(v_\mao)$ there exists an orthogonal matrix $\matrix{O}_u$ with $ u= \matrix{O}_u v_\mao$ and $\norm{\matrix{O}_u-\Id} < \epsilon$. {\em Source}: \cite[Lemma 15.1]{Mentemeier2013a}.

\Step[3]. Introduce the finite measure 
$$\tilde\eta(A \times B) ~:=~  \int_{B_{c/2}(\mao)} \, \1[A](\ma \as v_\mao)\, \1[B](-\log \abs{\ma v_\mao}) \, \llam^{d \times d}(d\ma).$$
Combining Steps 1 and 2 and Assumption \eqref{A4c}, there is $\delta >0$, such that for all $u \in B_\delta(v_\mao)$ there exists an orthogonal matrix $\matrix{O}_u$ with $ u= \matrix{O}_u v_\mao$ and $B_{c/2}(\mao)\matrix{O}_u \subset B_c(\mao)$. 
Hence for all $u \in B_\delta(v_\mao)$, by Assumption \eqref{A4c} and using that $\llam^{d \times d}$ is invariant under transformations by a matrix with determinant 1 (see \cite[proof of Prop. 15.2, Step 1]{Mentemeier2013a} for more details, using the Kronecker product)
\begin{align*}
\Prob(\mM^\top \as u \in A, -\log \abs{\mM^\top u} \in B) ~\ge&~ \gamma_0 \, \int_{B_{c/2}(\mao)\matrix{O}_u} \1[A](\ma \as u)\, \1[B](- \log \abs{\ma u}) \ \llam^{d \times d}(d\ma) \\
~=&~ \gamma_0 \, \int_{B_{c/2}(\mao)} \1[A](\ma \matrix{O}_u^{-1}\as u)\, \1[B](- \log \abs{\ma \matrix{O}_u^{-1} u}) \ \llam^{d \times d}(d\ma) \\
~=&~ \gamma_0\,  \tilde{\eta}(A \times B).
\end{align*}

\Step[4]: To obtain a minorization for the shifted measure $\Prob^\alpha_u$, recall that $\est[\alpha]$ is bounded from below and above, to obtain that
\begin{align*} \Prob_u^\alpha(U_1 \in A, Y_1 \in B) ~\ge&~ \int_{A \cap{B_\delta(v_\mao)}} \int_B \frac{\est[\alpha](w)}{\est[\alpha](u)} e^{-\alpha y} \ \Prob(\mM^\top \as u \in dw, -\log \abs{\mM^\top u} \in dy) \\
\ge&~ \gamma_1\int_{A \cap{B_\delta(v_\mao)}} \int_B {\est[\alpha](w)}e^{-\alpha y} \ \tilde{\eta}(dw,  dy) ~=:~ \eta(A \times B) 
\end{align*}
Upon renormalizing $\eta$ to a probability measure, and thereby determining $\gamma$, we obtain the assertion.
\end{proof}

Now we are ready to prove Lemma \ref{regenerationlemma} under Assumption \eqref{A4c}:

\begin{proof}[Proof of Lemma \ref{regenerationlemma}] Lemmata \ref{lem:min1} and \ref{lem:min2} imply that the chain $(U_n, Y_n)_{n \geq 0}$ is  $\Bigl(\mathcal{R}, \gamma, \eta, 1\Bigr)$-recurrent in the sense of \cite[Definition 2.2]{Athreya1978}. Then the lemma follows from \cite[Lemma 3.1 and Corollary 3.4]{Athreya1978}. The {\em regeneration times} $\sigma_n$ are constructed as follows: Let $(\xi_n)_{n \ge 0}$ be a sequence of i.i.d.~ Bernoulli(1,$\gamma$) random variables, independent of $(U_n, Y_n)_{n \ge 0}$. Whenever $(U_n, Y_n)$ enters the set $\mathcal{R}$, $(U_{n+1}, Y_{n+1})$ is generated according to $\eta$ if $\xi_n=1$, and according to $(1-\gamma)^{-1}(P-\gamma\eta)$ if $\xi=0$. The total transition probability thus remains $P=\Prob_u^\alpha((U_1, Y_1) \in \cdot)$. Together with Lemma \ref{lem:min1}, this construction immediately gives that $\sigma_1$ can be bounded stochastically by a random variable with geometric distribution.
\end{proof}

\section{Regular Variation of Fixed Points}\label{sect:regular variation}

In this section, we show that every fixed point of $\ST$, the existence of which is provided by Proposition \ref{prop:existence}, satisfies the regular variation property \eqref{eq:regvar}.

Let $\LTfp$ be the Laplace transform of a fixed point of $\ST$ in the critical case $m'(\alpha)=0$.
Introduce
\begin{align}\label{defn:D} D(u,t) :=\ & \frac{1-
\LTfp(e^{-t}u)}{e^{-\a t}\est[\a](u)}, \qquad u \in \Sp,\, t \in \R.
\end{align}
Our aim is to study behavior of $D$ as $t$ goes to infinity. Let $u_0$ be given by Corollary \ref{cor:uo}.
Following the approach in \cite{Kyprianou1998}, we are going to show that
$$ h_t(u,s) := \frac{D(u,s+t)}{D(u_0,t)} = \frac{1-\LTfp(e^{-(s+t)}u)}{e^{-\alpha s}(1-\LTfp(e^{-t}u_0))} \frac{\est[\alpha](u_0)}{\est[\alpha](u)}$$
converges to $1$ as $t$ tends to infinity. This shows in particular, that $D(u_0,t)$ is slowly varying as $t \to \infty$. We then use the results of \cite{Mentemeier2013} to deduce that this already implies that $D(u,t)$ is slowly varying for all $u \in \Sp$. 


\begin{lem}\label{lem:subsequence}
For every sequence $(t_k)_{k \in \N}$, tending to infinity, there is a subsequence $(t_n)_{n \in \N}$ such that $h_{t_n}(u,s)$ converges pointwise to a continuous function $h : \Sp \times \R \to [0, \infty)$.
\end{lem}

\begin{proof}
Introduce for $t \in \R$ the  function $f_t : \Rdnn \to [0, \infty)$
$$ f_t(x):= \frac{1-\LTfp(e^{-t}x)}{1-\LTfp(e^{-t}u_0)}.$$
Since $\LTfp$ is a Laplace transform and $t$ is fixed, it follows (using the multivariate version of the Bernstein theorem, \cite[Theorem 4.2.1]{Bochner1955}), that the derivative of $f_t$ is completely monotone in the multivariate sense, and hence,
$$ \varphi_t(x):= \exp(-f_t(x))$$
is the Laplace transform of a probability measure on $\Rdnn$, due to \cite[Criterion XIII.4.2]{Feller1971}. Note $\varphi_t(0)=1$, while the limit as $\abs{x} \to \infty$ may be positive, so the corresponding probability measure might have some mass in zero.

Since the set of probability measures is vaguely compact, we deduce that for any sequence $t_k$, tending to infinity, there is a subsequence $t_n$ such that $\varphi_{t_n}$ converges pointwise to the Laplace transform $\varphi$ of a (sub-)probability measure on $\Rdnn$, which is continuous except for maybe in $0$.  Since $\varphi_{t_n}(u_0)=e^{-1} > 0$ for all $n$, it follows that $\varphi >0$ on $\Rdnn$, and hence, we obtain that
$$ \lim_{n \to \infty} f_{t_n}(x) ~=~ f(x) ~:=~ - \log \varphi(x) $$ exists for all $x \in \Rdnn$ with $f$ being continuous on $\Rdnn \setminus\{0\}$. 

This implies the pointwise convergence
$$ \lim_{n \to \infty} h_{t_n}(u,s) ~=~ h(u,s) ~:=~ \frac{f(e^{-s}u)}{e^{-\alpha s}} \, \frac{\est[\alpha](u_0)}{\est[\alpha](u)},$$
where the function $h$ is continuous on $\R \times \Sp$. 
\end{proof}

\begin{lem}\label{lem:superharmonic}
Let $t_n$ be a sequence such that $h_{t_n}$ converges to a limit $h$. Then $h$ is superharmonic for $(U_n, V_n)$ under $\Prob_u^\alpha$, i.e.
$$ h(u,s) \ge \E_u^\alpha \, h(U_1, s+S_1).$$
\end{lem} 

\begin{proof}
Using Eq. \eqref{eq:STLT} and a telescoping sum, we obtain (since $\LTfp$ is a fixed point),
\begin{align*}
D(u,s+t) &=~\frac{1- \LTfp(e^{-(s+t)}u)}{e^{-\alpha(s+t)}\est[\alpha](u)} \\
 &=~ \Erw{\frac{1- \prod_{i=1}^N \LTfp(\mT_i^\top e^{-(s+t)}u)}{e^{-\alpha(s+t)}\est[\alpha](u)}} \\
 &=~ \Erw{\sum_{i=1}^N \frac{1- \LTfp(\mT_i^\top e^{-(s+t)}u)}{e^{-\alpha(s+t) }\est[\alpha](u)} \prod_{1 \le j < i} \LTfp(\mT_i^\top e^{-(s+t)}u)} \\
\end{align*}
Now divide by $e^{\alpha t}(1-\LTfp(e^{-t}u_0)) / \est[\alpha](u_0)$ to obtain
\begin{align*}
h_{t}(u,s)=&~ \frac{\est[\alpha](u_0)}{\est[\alpha](u)} \E\Biggl(\sum_{i=1}^N \frac{1- \LTfp(e^{-S^u(i)-(s+t)}U^u(i))}{(1-\LTfp(e^{-t}u_0))\est[\alpha](U^u(i)) e^{-\alpha S^u(i)}e^{-\alpha s}}e^{-\alpha S^u(i)} \est[\alpha](U^u(i))\\& \hspace{8cm}\times\prod_{1 \le j < i} \LTfp(e^{-S^u(i)-(s+t)}U^u(i))\Biggr) \\
=&~ \frac{\est[\alpha](u_0)}{\est[\alpha](u)} \E\Biggl(\sum_{i=1}^N \  \, \frac{f_{t}\Bigl( e^{-S^u(i)-s}, U^u(i) \Bigr)}{(\est[\alpha](U^u(i)) e^{-\alpha (S^u(i)+s)}}e^{-\alpha S^u(i)} \est[\alpha](U^u(i)) \\& \hspace{8cm}\times\prod_{1 \le j < i} \LTfp(e^{-S^u(i)-(s+t)}U^u(i))\Biggr) \\
=&~ \frac{1}{\est[\alpha](u)} \E\Biggl(\sum_{i=1}^N \  \, h_t\Bigl(U^u(i), s+S^u(i) \Bigr) \, e^{-\alpha S^u(i)} \est[\alpha](U^u(i)) \prod_{1 \le j < i} \LTfp(e^{-S^u(i)-(s+t)}U^u(i))\Biggr) 
\end{align*}
Now consider the subsequential limit $t_n \to \infty$, then the LHS converges by assumption to $h$, while for the RHS, we use Fatou's lemma and observe that the product tends to $1$, so that we obtain:
\begin{align*}
h(u,s) \ge &~ \frac{1}{\est[\alpha](u)} \Erw{\sum_{i=1}^N \  \, h\Bigl(U^u(i), s+S^u(i) \Bigr) \, e^{-\alpha S^u(i)} \est[\alpha](U^u(i)) } \\
=&~ \E_u^\alpha \, h(U_1, s+S_1).
\end{align*}
\end{proof}

\begin{lem}\label{lem:limit}
The (subsequential limit) function $h$ is constant and equal to 1 on $\supp \pist[\alpha] \times \R$.
\end{lem}

\begin{proof}
It follows from Lemma \ref{lem:superharmonic} that $h(U_n,s+ S_n)$ is a nonnegative supermartingale, which hence converges a.s. as $n \to \infty$. 
Now assume that $h(u,s) \neq h(w,t)$ for $u,w \in \supp \pist[\alpha]$ and $s,t \in \R$. Since $m'(\alpha)=0$, $(U_n, S_n)$ under $\Prob_{u_0}^\alpha$ is a recurrent Markov Random Walk by Lemma \ref{lem:recurrence}, thus it visits every neighborhood of $(u,s)$ resp. $(w,t)$ infinitely often. But then, due to the a.s. convergence of $h(U_n, s+S_n)$ and the continuity of $h$, we infer that $h$ has to be constant.
Since furthermore $h(u_0,0)=1$, the assertion follows. 
\end{proof}

\begin{rem}\label{rem:aperiodic}
Note that here (via Lemma \ref{lem:recurrence}) the aperiodicity condition enters. It is not needed if $\alpha=1$, because then $h$ itself is a multivariate Laplace transform, which is in particular monotone. Then using again the a.s.~convergence of $h(U_n, s+S_n)$ together with the fact that $S_n$ oscillates (see Eq. \eqref{eq:oscillates}) shows that $h$ has to be constant. 
\end{rem}

\begin{lem}\label{lem:slowvar}
 It holds that 
\begin{equation}\label{eq:convh} \lim_{t \to \infty} \frac{1-\LTfp(e^{-(s+t)}u)}{e^{-\alpha s}(1-\LTfp(e^{-t}u_0))} \frac{\est[\alpha](u_0)}{\est[\alpha](u)}~=~1 \qquad \forall u \in \Sp, \, s \in \R,\end{equation}
and the convergence is uniform on compact subsets of $\Sp \times \R$.
In particular, the positive function 
\begin{equation} \label{eq:defL} L(r) ~:=~ \frac{1-\LTfp(ru_0)}{r^\alpha \est[\alpha](u_0)} \qquad \bigg( =~ D(u_0, - \log r) \bigg) \end{equation}
is slowly varying at 0, and 
\begin{equation}\label{eq:slowvar} \lim_{r \to 0} \, \sup_{u \in \Sp} \abs{\frac{1-\LTfp(ru)}{L(r) \, r^\alpha} - H^\alpha(u)} ~=~ 0. \end{equation}
\end{lem}

\begin{proof}
Combining Lemmata \ref{lem:superharmonic} and \ref{lem:limit}, we obtain that for every sequence $t_k \to \infty$ there is a subsequence $t_n \to \infty$ such that for each $s \in \R$,
$$ 1 ~=~ \lim_{n \to \infty} h_{t_n}(u_0,s) ~=~  \lim_{n \to \infty} \frac{1-\LTfp(e^{-(s+t_n)}u_0)}{e^{-\alpha s}(1-\LTfp(e^{-t_n}u_0))}.$$
Since all subsequential limits are the same, we infer that $\lim_{t \to \infty} h_t(u_0,s)=1$ for all $s \in \R$, which in particular proves the slow variation assertion about the function $L(r)$, for $L(sr)/L(r) = h_{-\log r}(u_0, -\log s)$.
Using the estimate
$$ \left( \min_{1 \le i \le d} (u_0)_i \right) (1-\LTfp(r\deins)) ~\le~(1-\LTfp(ru_0)) ~\le~ (1-\LTfp(r\deins))$$
(see \cite[Lemma A.1]{Mentemeier2013}), we deduce further that
$$ 0 ~<~ \liminf_{r \to \infty} \frac{1-\LTfp(r\deins)}{L(r)r^\alpha} ~\le~ \limsup_{r \to \infty} \frac{1-\LTfp(r\deins)}{L(r)r^\alpha} ~<~ \infty,  $$
i.e., $\LTfp$ is {\em $L$-$\alpha$-regular} in the sense of \cite[Definition 2.1]{Mentemeier2013}. Then \cite[Theorem 8.2]{Mentemeier2013}  provides us with the first assertion, i.e. the (uniform) convergence in Eq. \eqref{eq:convh}. Then Eq. \eqref{eq:slowvar} is a direct consequence when considering the compact set $\Sp \times \{0\}$.
\end{proof}

\section{Uniqueness of Fixed Points}\label{sect:uniqueness}

In this section, we are going to finish the proof of Theorem \ref{thm:main}. Therefore, we show that the slowly varying function appearing in \eqref{eq:regvar} is essentially unique, and that this property then identifies the fixed points. The approach is the multivariate analogue of \cite[Theorem 8.6]{Biggins1997}.

We start with the following lemma, the proof of which we postpone to the end of this section for a better stream of arguments.

\begin{lem}\label{lem:maxpos}
Assume \eqref{A1}--\eqref{A3}, \eqref{A5} and \eqref{A6}. Then
\begin{equation}\label{eq:maxpos}
 \lim_{n \to \infty} \max_{\abs{v}=n} \norm{\mL(v)}=0 \qquad \Pfs
\end{equation}
\end{lem}

For $u \in \Sp$, we can introduce for $t \in \R$ the {\em homogeneous stopping line}
$$ \slineu[t] ~:=~ \left\{ v \in \tree \, : \, S^u(v) > t, \ S^u(v|k) \le t \, \forall k < \abs{v} \right\}. $$
Since $\max_{\abs{v}=n} \norm{\mL(v)} \to 0$ $\Pfs$ by Lemma \ref{lem:maxpos}, this stopping line is finite $\Pfs$ and intersects the whole tree ({\em is dissecting}). 

Let $\LTfp$ be a fixed point of $\ST$. Define 
$$ M_n(x) := \prod_{\abs{v}=n} \, \LTfp(\mL(v)^\top x), \qquad x \in \Rdnn.$$
By Eq. \eqref{eq:STLT}, this constitutes a bounded martingale w.r.t.~$\B_n$ for every $x$ and we call its $\Pfs$ limit $M(x) \in [0, \infty)$ the {\em disintegration} of the fixed point $\LTfp$.
Setting
$$ Z(x) ~:=~ - \log M(x),$$ the martingale property together with  boundedness implies that $\LTfp(x)=\E \exp(-Z(x))$ for all $x \in \Rdnn$.
{Following the proof of \cite[Lemma 4.1]{AM2010a}, one can show that $M(\cdot,\omega)$ is a Laplace transform for $\Prob$-a.e. $\omega \in \Omega$, and that $M$ is jointly measurable on $\Sp \times \Omega$. This implies the same for $Z$.}

\begin{prop}\label{prop:dis}
Assume \eqref{A1}-- \eqref{A6} { and \eqref{A8}}. Let $\LTfp$ be a fixed point of $\ST$ with disintegration $M$. Let $F : \Rdnn \to [0,\infty)$ be a nonnegative measurable function with $\lim_{s \to 0} \sup_{u \in \Sp} \abs{F(su)-\gamma}=0$ for some $\gamma \ge 0$. Then the following holds:
\begin{enumerate}
\item $\lim_{n \to \infty} \, \sum_{\abs{v}=n} F(\mL(v)^\top x) (1 - \LTfp(\mL(v)^\top x)) = \gamma Z(x)$ $\Pfs$\label{as2}
\item For all $u \in \Sp$, $r \in \Rp$, $Z(ru) = r^\alpha Z(u)$.\label{as3}
\item $\LTfp(ru) = \E e^{-r^\alpha Z(u)}$ for all $u \in \Sp$, $r \ge 0$.\label{as4}
\item $ Z(u) \in (0, \infty)$ $\Pfs$.\label{as5}
\item \label{eq:Mt} $\lim_{t \to \infty} \, \sum_{v \in \slineu[t]} \left(1 - \LTfp(e^{-S^u(v)} U^u(v)) \right) ~=~ Z(u)$  $\Pfs$ for all $u \in \Sp$.
\end{enumerate}
\end{prop}

\begin{proof}
Using Lemma \ref{lem:maxpos}, the proof of Assertion \eqref{as2}  is the same as for \cite[Lemma 7.3]{Mentemeier2013} and therefore omitted.
By Lemma \ref{lem:slowvar}, for all $r \in \Rp$ and $u \in \Sp$, the function $F(su):= \frac{1- \LTfp(rsu)}{1-\LTfp(ru)}$ converges uniformly to $r^\alpha$. 
Thus we obtain \eqref{as3} by an application of \eqref{as2}. Then \eqref{as4} is an immediate consequence of $\LTfp(x)=\E \exp(-Z(x))$.

Reasoning as in the proof of \cite[Theorem 2.7, Step 6]{BDGM2014}, we see that for any nontrivial fixed point $X$ of $\ST$, $\P{X=0}=0$, and consequently $Z(u) > 0$ $\Pfs$. On the other hand, since $\LTfp$ is the Laplace transform of a random variable on $\Rdnn$, $Z(u) < \infty$ $\Pfs$
\end{proof}



The subsequent lemma is where we use assumption \eqref{A7}. Using the definition of $\mu$, it implies that with $c' := - \log c$
\begin{align*} \P{S^u(i) > c' \ \forall\, 1 \le i \le N} ~\le&~ \E \bigg[ \sum_{i=1}^N \1\big( S^u(i) > c' \big) \bigg] \\
=&~ (\E N) \E \bigg[ \1\big( -\log \abs{\mM^\top u} > c' \big)  \bigg] ~=~ (\E N) \P{\abs{\mM^\top u} < c} \\
\le&~ (\E N) \P{\iota(\mM^\top) < c} ~=~0.
\end{align*}
In other words, the increments of $S(vi) - S(v)$ are $\Pfs$ bounded by $c'$.

\begin{lem}\label{lem:LMt}
Assume \eqref{A1}--\eqref{A8}. 
Let $\LTfp$ be a nontrivial fixed point of $\ST$ with associated slowly varying function $L$ given by Eq. \eqref{eq:defL}.
Then 
\begin{equation}
\label{eq:LMt} \lim_{t \to \infty} L(e^{-t}) \, \sum_{v \in \slineu[t]} \est[\alpha](U^u(v)) e^{- \alpha S^u(v)} ~=~ Z(u) \quad \Pfs[u]
\end{equation}
\end{lem}

\begin{proof}
By Lemma \ref{lem:slowvar}, 
$$ \lim_{t \to \infty} \frac{1 - \LTfp(e^{-s-t}y)}{\est[\alpha](y) e^{-\alpha(s+t)} L(e^{-t}) } = 1,$$
and the convergence is uniform on compact sets for $(y,s)$. In particular, it is uniform on the set $\Sp \times [0,c']$. Now applying this result with $s=S^u(v)-t$ and $y = U^u(v)$ with $v \in \slineu[t]$ and using that $$0 ~<~ S^u(v)-t ~\le~ S^u(v)-S^u(v|(\abs{v}-1))   \in [0,c']$$ by Assumption \eqref{A7}, we deduce from Proposition \ref{prop:dis}, \eqref{eq:Mt} that
\begin{align*}
Z(u) ~=&~ \lim_{t \to \infty} \, \sum_{v \in \sline[t]} L(e^{-t}) \est[\alpha](U^u(v)) e^{- \alpha S^u(v)} \frac{1 - \LTfp(e^{-(S^u(v)-t)-t}U^u(v))}{\est[\alpha](U^u(v)) e^{-\alpha(S^u(v)-t+t)} L(e^{-t}) }\\
=&~ \lim_{t \to \infty} L(e^{-t}) \, \sum_{v \in \sline[t]} \est[\alpha](U^u(v)) e^{- \alpha S^u(v)} \qquad \Pfs
\end{align*}
\end{proof}

\begin{rem}
The idea of this proof follows that of \cite[Theorem 8.6]{Biggins1997}. There an assumption similar to \eqref{A7} is avoided by using the theory of general branching processes, see \cite{Jagers1975,Nerman1981}. A similar approach is taken in \cite{Mentemeier2013} in the non-critical case, a crucial ingredient of which is an application of Kesten's renewal theorem \cite[Theorem 1]{Kesten1974}. In the critical case, a variant of Kesten's renewal theorem for driftless Markov random walks, or a strong theory of Wiener-Hopf factorization seems to be needed in order to proceed along similar lines.
\end{rem}

Now we are ready to prove our main result.


\begin{proof}[Proof of Theorem \ref{thm:main}]
\Step[1]: By Proposition \eqref{prop:existence}, there is a nontrivial fixed point of $\ST$ with LT $\LTfp$, say. By Proposition \ref{prop:dis}, for each $u \in \Sp$, there is a random variable $Z(u)$ with $\P{Z(u)>0}=1$ and such that $\LTfp(ru)=\E [\exp(-r^\alpha Z(u)) ]$ for all $r \in [0,\infty)$. 
Define $L(r)$ by \eqref{eq:defL}, choosing a suitable $u_0$.
\medskip 

\Step[2]: Let now $\LTfp_2$ be the Laplace transform of a different nontrivial fixed point, with corresponding disintegration $M_2$ and $Z_2$, and slowly varying function $L_2$, defined by \eqref{eq:defL}, using the same $u_0$ as before. Recall that $Z(u)$ and $Z_2(u)$ are $\Pfs$ positive and finite by by Proposition \ref{prop:dis}, \eqref{as5} for each $u \in \Sp$. 
%
Then we have by Lemma \ref{lem:LMt} that $\Pfs$,
$$ \lim_{t \to \infty} \frac{Z_2(u)}{Z(u)} = \lim_{t \to \infty} \frac{L_2(e^{-t}) \, \sum_{v \in \slineu[t]} \est[\alpha](U^u(v)) e^{- \alpha S^u(v)} }{L(e^{-t}) \, \sum_{v \in \slineu[t]} \est[\alpha](U^u(v)) e^{- \alpha S^u(v)} } = \lim_{t \to \infty} \frac{L_2(e^{-t})}{L(e^{-t})}.$$
First, fixing $u \in \Sp$, this proves that the limit of the right hand side exists and equals some $K \in (0, \infty)$. Then, using  the equation again for general $u$, we obtain  $Z_2(u) = K Z(u)$ $\Pfs$. Consequently,
$$ \LTfp_2(ru)= \Erw{e^{-r^\alpha Z_2(u)}} = \Erw{e^{-r^\alpha K Z(u)}} = \LTfp(K^{1/\alpha}ru),$$
which proves Eq. \ref{LTofFP}.

\medskip

\Step[3]: Fix $L$ to be the slowly varying function corresponding to $\LTfp$. Then Eq. \eqref{eq:regvar} follows from Eq. \eqref{eq:slowvar} for this particular $\LTfp$, and moreover, 
$$ \lim_{r \to 0} \frac{1-\LTfp_2(ru)}{r^\alpha L(r)} ~=~ \lim_{r \to 0} \frac{K(1-\LTfp(K^{1/\alpha}ru))}{K r^\alpha L(K^{1/\alpha}r)} \frac{L(K^{1/\alpha}r)}{L(r)} ~=~ K H^\alpha(u)$$
The final assertion about $\limsup_{r \to 0} L(r)$ will be proved in Lemma \ref{lem:Linfty}.
\end{proof}

\subsection{Proof of Lemma \ref{lem:maxpos}}

Using Proposition \ref{prop:many to one}, one shows that for all $u \in \Sp$,
$$ W_n(u) ~:=~ \sum_{\abs{v}=n} H^\alpha(\mL(v)^\top u) ~=~  \sum_{\abs{v}=n} \int_{\Sp} \, \skalar{\mL(v)^\top u,y}^\alpha \, \nus[\alpha](dy)$$
defines a nonnegative martingale w.r.t.~the filtration $\B_n$. Its $\Pfs$ limit $W(u)$ appears prominently in the non-critical case, where every fixed point has a Laplace transform of the form $\LTa(ru)=\E \exp(-K r^\alpha W(u))$, see \cite[Theorem 1.2]{Mentemeier2013}.
In the critical case, its limit is trivial:

\begin{prop}\label{prop:W}
Assume \eqref{A1}--\eqref{A3} and \eqref{A5} and \eqref{A6}. Then $W(u) = 0$ $\Pfs$ for all $u \in \Sp$.
\end{prop} 

\begin{proof}
Since $W(u)$ as the limit of a nonnegative martingale is again nonnegative, it suffices to show that $\E W(u)=0$. It even suffices to show that $\E W(u_0)=0$ for one $u_0 \in \interior{\Sp}$, for due to nonnegativity 
\begin{equation}\label{eq:eq1}\skalar{ u,\mL(v) y} ~\le~ \skalar{ \deins, \mL(v) y} ~\le~ \frac{1}{\min_i \, (u_0)_i} \skalar{u_0, \mL(v) y}
\end{equation} and hence $W_n(u) \le c W_n(u_0)$ for $c=[\min_{i} (u_0)_i]^{-1}$. 

It is shown in \cite[Theorem 2.1 (iii)]{Biggins2004}, that $\E W(u_0)=0$ follows from
 $\limsup_{n \to \infty} \est[\alpha](U_n)e^{\alpha S_n} = \infty$ $\Prob_{u_0}^\alpha$-a.s. But the latter is a direct consequence of \eqref{eq:oscillates2}, together with the strict positivity of $\est[\alpha]$. 
\end{proof}




\begin{proof}[Proof of Lemma \ref{lem:maxpos}]
Let as before $u_0 \in \interior{\Sp}$ and set $c=[\min_{i} (u_0)_i]^{-1} < \infty$. Recalling Eq. \ref{eq:eq1} and the definition of $W_n(u_0)$, we have 
$$c W_n(u) ~\ge~ \sum_{\abs{v}=n} \int_{\Sp} \abs{\mL(v)y}^\alpha \, \nus[\alpha](dy).$$ By \cite[Corollary 4.7]{BDGM2014}, there is a  constant $C$ such that for any allowable $\ma$, $\norm{\ma}^\alpha \le C \int_{\Sp} \, \abs{\ma y}^\alpha \, \nus[\alpha](dy),$ hence $$ \frac{C}{c_u} \sqrt{d} W_n(\eins) ~\ge~ \sum_{\abs{v}=n} \norm{\mL(v)}^\alpha ~\ge~ \max_{\abs{v}=n}\, \norm{\mL(v)}^\alpha,$$
and the assertion follows.
\end{proof}

%
%
%
%
%

\begin{lem}\label{lem:Linfty}
Under the assumptions of Theorem \ref{thm:main}, $\limsup_{r \to \infty} L(r)=0$. If $\alpha=1$, then $\E\abs{X}=\infty$ for every nontrivial fixed point $X$.
\end{lem}

\begin{proof}
Suppose that $\limsup_{r \to 0} L(r) \le C < \infty$. By an extension of Prop. \ref{prop:dis}, \eqref{as2}, 
\begin{align*} Z(u) ~=&~  \lim_{n \to \infty} \sum_{\abs{v}=n} L(\abs{\mL(v)^\top u}) \,  H^\alpha(\mL(v)^\top u) \frac{1-\LTfp(\mL(v)^\top u)}{L(\abs{\mL(v)^\top u}) \,  H^\alpha(\mL(v)^\top u)} \\ ~\le&~ C \lim_{n \to \infty}  \sum_{\abs{v}=n} H^\alpha(\mL(v)^\top u) ~=~ C W(u) =0 \end{align*}
by Proposition \ref{prop:W}, which gives a contradiction.

If now $\alpha=1$, then $$ \lim_{r \to 0} \frac{1-\LTfp(ru)}{r} ~=~ \skalar{u, \E X},$$
being finite or not. Combining this with Eq. \eqref{eq:regvar} implies that
$$ \lim_{r \to 0} L(r) ~=~ \frac{\skalar{u, \E X}}{K \est[\alpha](u)},$$ hence $\E \abs{X} =\infty$, since $\limsup_{r \to 0} L(r)=\infty$.
\end{proof}

\section{Determining the Slowly Varying Function}\label{sect:L}

In this section, we work under one of the additional assumptions \eqref{A4c} or \eqref{A4f}, together with \eqref{A8}. 
We want to identify the slowly varying function $L$, which was (given a nontrivial fixed point $\LTfp$ and a reference point $u_0 \in \supp \nust[\alpha]$) defined in Eq. \eqref{eq:defL} to be
$$ L(r) ~=~ \frac{1-\LTfp(ru_0)}{r^\alpha H^\alpha (u_0)} ~=~ D(u_0, - \log r).$$
We are going to show that
\begin{equation}\label{eq:Dsv} \lim_{t \to \infty} \frac{D(u_0, t)}{t} = K' \in (0,\infty), \end{equation}
which gives that $\lim_{r \to 0}L(r)/\abs{\log r} =K'$, i.e. we may choose the slowly varying function to be a scalar multiple of $\abs{\log r} \vee 1$.

The basic idea to prove Eq. \eqref{eq:Dsv} comes from \cite{DL1983} and is by using a renewal equation satisfied by (the one-dimensional analogue of) $D(u_0,t)$. In the present multivariate situation, we obtain a Markov renewal equation for a drift-less Markov random walk. By a clever application of the regeneration lemma, we can reduce this again to a (one-dimensional) renewal equation for a drift-less random walk, for which enough theory is known to solve it.

\subsection{The Renewal Equation} In this subsection we present the Markov renewal equation for $D(u,t)$ and show how, using Lemma \ref{regenerationlemma}, it can be replaced by a one-dimensional renewal equation.
\begin{lem}\label{lem:link_D_G}
Assume \eqref{A1}--\eqref{A3} and \eqref{A5}. Then the following renewal equation holds 
\begin{equation}
\label{eq:renewal_D}
D(u,t) =  \E_u^\a D(U_1, t+ S_1) - G(u,t),
\end{equation}
where
\begin{align}
\label{defn:G} G(u,t) :=\ & \frac{e^{\a
t}}{\est[\a](u)}\Erw{\prod_{i=1}^N \LTa(e^{-t}\mT_i^\top u) + \sum_{i=1}^N\left(
1- \LTa(e^{-t}\mT_i^\top u) \right) -1}.
\end{align}
\end{lem}

\begin{proof}[Source:] Lemma 9.6 in \cite{Mentemeier2013a}, note there the different notation $V_1=-S_1$. \end{proof}

\begin{lem}\label{lem:properties_of_G}
Assume \eqref{A1}--\eqref{A3} and \eqref{A5}.Then
\begin{enumerate}
  \item $G(u,t) \ge 0$ for all $(u,t) \in \Sp \times \R$.
  \item For all $u \in \Sp$, $t \mapsto e^{-\alpha t}G(u,t)$ is
  decreasing. 
\end{enumerate}
\end{lem}

\begin{proof}[Source:]
Lemma 9.7 in \cite{Mentemeier2013a}, being a straightforward generalization of \cite[Lemma
2.4]{DL1983}.
\end{proof}

From now on, assume that the assumptions of the Regeneration Lemma, Lemma \ref{regenerationlemma} are satisfied, i.e.~there is a sequence of stopping times $(\sigma_n)_{n \in \N}$ and a probability measure $\eta$ on $\Sp \times \R$ such that in particular \eqref{R3}
 holds. 
 
 For any nonnegative measurable function $F$ on $\Sp\times\R$ we define $\hat{F}:\R\mapsto\R$ by 
\begin{align}
 \label{def:projection}
 \hat{F}(t) ~:=~\E_\eta^\alpha \, {F(U_{\sigma_1-1},t+S_{\sigma_1-1})}.
\end{align}
Moreover, under each $\Prob_u^\alpha$, let $(V_n)_{n \in \N}$ be a zero-delayed random walk with increment distribution $\Prob_\eta^\alpha(S_{\sigma_1-1} \in \cdot)$, independent of all other occurring random variables. Note that $V_n$ is a drift-less random walk.
\begin{lem}
 For any nonnegative measurable function $F$ on $\Sp\times\R$ and $k\ge0$, the following equation holds
 $$\E_\eta^\alpha \left[F(U_{\sigma_{k+1}-1}, S_{\sigma_{k+1}-1})\right]=\E_\eta^\alpha {\hat{F}(V_k)}$$
\end{lem}
\begin{proof}
We prove by induction. By the definition of $\hat{F}$, the equation holds for $k=0$. 
 Suppose now that it holds for some $k \ge 0$. Then
 \begin{align*}
  \Ex[\eta]{F(U_{\sigma_{k+2}-1}, S_{\sigma_{k+2}-1})}~=~\Ex[\eta]{\E_\eta^\alpha \bigg[F(U_{\sigma_{k+2}-1},  S_{\sigma_{1}-1}+(S_{\sigma_{k+2}-1}- S_{\sigma_{1}-1}))|\F_{\sigma_{1}-1}\bigg]}\\
  ~=~\Ex[\eta]{{\E^\alpha_{\eta}}' \bigg[F(U'_{\sigma_{k+1}-1},  S_{\sigma_{1}-1}+S'_{\sigma_{k+1}-1})\bigg]}
  ~=~\Ex[\eta]{{\E_\eta^\alpha}' \bigg[{\hat{F}(S_{\sigma_{1}-1}+V'_k)}\bigg]}=\Ex[\eta]{\hat{F}(V_{k+1})},
 \end{align*}
where \eqref{R3} from Lemma \ref{regenerationlemma} is used in the second equality and we denote by $(U'_n,S'_n), V_k$ an independent copy of $(U_n,S_n), V_k$ with corresponding expectation ${\E_\eta^\alpha}'$.
\end{proof}

Now we can formulate the univariate renewal equation, corresponding to Eq. \eqref{eq:renewal_D}.

\begin{lem}
\label{lem:renewal_hatD}
 For  $g(t)=\Ex[\eta]{\sum_{i=0}^{\s_1-2}G(U_i,t+V_1+S_i)}$ we have
 \begin{equation}
  \label{eq:renewal_hatD}
  \hat D(t)=\E_\eta^\alpha\hat D(t+V_1)-g(t). 
 \end{equation}
\end{lem}
\begin{proof}
 Let $$M_n=D(U_n,t+S_n)-\sum_{i=0}^{n-1}G(U_i,t+S_i).$$
Since $(U_n,S_n)$ is a Markov chain, the Markov renewal equation \eqref{eq:renewal_D} implies that $M_n$ is a $\Prob_u^\alpha$-martingale (with respect to the filtration $\mathcal{G}_n$) for each $u \in \supp \nust[\alpha]$. Since $\tau=\sigma_1-1$ is a stopping time by \eqref{R1}, the optional stopping theorem implies that 
 
\begin{equation}\label{eq:a}D(u,t+s)=\Ex[(u,s)]{D(U_{\s_1-1},t+S_{\s_1-1})-\sum_{i=0}^{\s_1-2}G(U_i,t+S_i)}\end{equation} and 
\begin{equation}\label{eq:c}  D(u,t+s)=\Ex[(u,s)]{D(U_{\s_2-1},t+S_{\s_2-1})-\sum_{i=0}^{\s_2-2}G(U_i,t+S_i)}. \end{equation} Equating the right hand sides of \eqref{eq:a} and \eqref{eq:c} and integrating with respect to $\eta$, we obtain
  \begin{align*}
     \hat D(t)=\Ex[\eta]{D(U_{\s_1-1},t+S_{\s_1-1})}=\Ex[\eta]{D(U_{\s_2-1},t+S_{\s_2-1})-\sum_{i=\s_1-1}^{\s_2-2}G(U_i,t+S_i)}\\
     =\E_\eta^\alpha \hat D(t+V_1)-\Ex[\eta]{\sum_{i=0}^{\s_1-2}G(U_i,t+V_1+S_i)}.
  \end{align*}
\end{proof}

  \subsection{Solving the Renewal Equation}
  In this subsection, we will show that $\lim_{t \to \infty} D(u_0,t)/t=1$. Before we can use the renewal equation, we first have to consider some technicalities, e.g. direct Riemann integrability of $g$. We start by considering  moments of $V_1$.
  \begin{lem}
   \label{lem:exp_moment}
   Assume additionally \eqref{A7}-\eqref{A8}. Then there exists $\delta>0$ such that $\E_\eta^\alpha {e^{\delta |V_1|}}<\8$. 
  \end{lem}
  \begin{proof}  We proof the boundedness of $\E_\eta^\alpha {e^{-\delta V_1}}$ and $\E_\eta^\alpha {e^{\delta V_1}}$ separately, starting with the first one.
  
%
%
%
Property (R4) implies that  there exists $\delta_0$ such that $\sup_u\Ex[u]{e^{\delta_0 (\sigma-1)}}<\8$.
   
   
   Due to Assumption \eqref{A8}, there is $\eps>0$ such that $m(\alpha+\eps)\le e^{\delta_0}$.  
   Observe that there is $C_\epsilon <\infty$ such that 
   $$  \frac{e^{-\eps S_n}}{m(\alpha+\eps)^n} ~\le~ C_\eps \frac{\est[\alpha](u)}{\est[\alpha+\eps](u)} \frac{\est[\alpha+\eps](U_n)}{\est[\alpha](U_n)}  \frac{e^{-\eps S_n}}{m(\alpha+\eps)^n}, $$
   and the right hand side is a martingale under $\Prob_u^\alpha$ with expectation $C_\epsilon$ due to Proposition \ref{prop:many to one}.
Therefore, the 
   optional stopping theorem and the Fatou lemma  imply
   $$\E_u^\alpha \left( \frac{e^{-\eps S_{{\sigma-1}}}}{m(\alpha+\eps)^{{\sigma-1}}} \right) \le\lim_{n \to \infty}\E_u^\alpha \left( \frac{e^{-\eps S_{(\sigma-1)\wedge n}}}{m(\alpha+\eps)^{(\sigma-1)\wedge n}} \right) \le C_{\eps}.$$
   The choice of $\eps$ gives us $\sup_u\Ex[u]{m(\alpha+\eps)^{\sigma-1}}<\8$, hence  by the Cauchy-Schwartz inequality,
   $$(\Ex[u]{e^{-\frac{\eps}{2} S_{\sigma-1}}})^2 \le\Ex[u]{e^{-\eps S_{\sigma-1}}/m(\alpha+\eps)^{\sigma-1}} \Ex[u]{m(\alpha+\eps)^{\sigma-1}}$$
   is bounded uniformly in $u$. Choose $\delta= \min\{\delta_0,\epsilon/2\}$.
   
   For the second part recall that  assumption \eqref{A7} implies that the increments of $S_n$ are bounded from above by $-\log c$. Therefore, $$\sup_u\Ex[u]{e^{\delta S_{\sigma-1}}}\le \sup_u\Ex[u]{{(1/c)^{\delta_0 (\sigma-1)}}}<\8.$$
   Integrating with respect to $\eta$ finishes the proof.
  \end{proof}
  
  Before proving that $g(t)$ is dRi, we need the following consequence of the slow variation of $D(u_0,t)$ (for $t \to \infty$).

  \begin{lem}  \label{lem:bounds_on_LTa}
Let $d^*(t)=\sup_{u\in \Sp} D(t,u)$. Then for all $0<\eps<\alpha$, there is $C >0$, such that for $t\ge0$ and any $s$ 
\begin{align} \label{eqn:bound_1-LTa}
d^*(s)&\le Ce^{\eps s}, \\
\label{eqn:bound_1-LTa_2}
\frac{d^*(t+s)}{L(e^{-t})}&\le C e^{\eps |s|} .
\end{align}

\end{lem}

\begin{proof}
Since the ratio $D(t,u)/L(e^{-t})$ is bounded it suffice to show the above inequalities with $L(e^{-t})$ instead of $d^*(t)$. 
Potter's theorem \cite[Theorem 1.5.6]{BGT1987}, applied to the slowly varying function 
$L$ proves that 
\begin{align}
 \label{eq:potter}
 \frac{L(e^{-x})}{L(e^{-y})}\le C e^{\eps |x-y|},
\end{align}
for any positive $x,y$. Using also the trivial bound $L(e^{-t})\le C e^{\alpha t}$  we get \eqref{eqn:bound_1-LTa}.
In order to show \eqref{eqn:bound_1-LTa_2} we use \eqref{eq:potter} in the case when $t+s\ge0$. When $t+s\le0$ we have
\begin{align*}
 \frac{L(e^{-t-s})}{L(e^{-t})}=\frac{L(e^{-t-s})}{L(1)}\frac{L(1)}{L(e^{-t})}\le Ce^{\alpha(t+s)}e^{\eps t}\le Ce^{\eps |s|}.
\end{align*}

\end{proof}

  \begin{lem}
  \label{lem:g_is_dRi}
    Assume in addition \eqref{A7} and \eqref{A8}. Then the function $g(x)$ is nonnegative and directly Riemann integrable.
  \end{lem}
  \begin{proof} 
Referring to  Lemma \ref{lem:properties_of_G}, $G$ is nonnegative and $t \mapsto e^{-\alpha t} {G}(t)$ is
decreasing, hence the same holds for $g$. For such functions, a sufficient condition for direct Riemann integrability is 
that ${g} \in \Lp[1]{\R}$, see \cite[Lemma 9.1]{Goldie1991}. 
Since moreover, by Lemma \ref{regenerationlemma}, $\E\,\sigma_1<\8$,  it suffices to show the integrability of ${g}^* : t \mapsto \sup_{u \in \Sp} G(u,t)$.  

Set $h(x):=e^{-x}+x-1$. Since $h$ is positive for $x\ge0$, we have $\LTa(e^{-t}\mT_i^\top u)\le e^{(1-\LTa(e^{-t}\mT_i^\top u))}$. Therefore
\begin{align*}
 \int g^*(t)dt~=&~\int \sup_{u \in \Sp} \frac{e^{\a
t}}{\est[\a](u)}\Erw{\prod_{i=1}^N \LTa(e^{-t}\mT_i^\top u) + \sum_{i=1}^N\left(
1- \LTa(e^{-t}\mT_i^\top u) \right) -1}dt \\
\le&~ C\int \sup_{u \in \Sp} {e^{\a t}}\Erw{e^{\sum_{i=1}^N (1-\LTa(e^{-t}\mT_i^\top u))} + \sum_{i=1}^N\left(1- \LTa(e^{-t}\mT_i^\top u) \right) -1}dt\\
=&~ C\int \sup_{u \in\Sp} {e^{\a t}}\Erw{h\left(\sum_{i=1}^N(1-\LTa(e^{-t}\mT_i^\top u)) \right)} dt.
\end{align*}
Using Lemma \ref{lem:bounds_on_LTa}, boundedness of $H^{\alpha}$ and fact that $h(x)$ is increasing,  comparable with $\min(x,x^2)$  on the positive half line, the later can be bounded by 
\begin{align*}
\int \sup_{u \in\Sp} \, {e^{\a t}}\Erw{h\left(\sum_{i=1}^Ne^{(\eps-\alpha)t}\|\mT_i^\top u\|^{\alpha-\eps} \right)} dt
\le~ C\,\Erw{\int {e^{\a t}}h\left(e^{(\eps-\alpha)t}\sum_{i=1}^N\|\mT_i\|^{\alpha-\eps} \right) dt} \\
\le~ C\,\Erw{\left(\sum_{i=1}^N\|\mT_i\|^{\alpha-\eps}\right)^{\frac{\alpha}{\alpha-\eps}}\int {e^{ \frac{\a }{\alpha-\eps}s}}h\left(e^{-s} \right) ds}
<\8,
\end{align*}
 by \eqref{eq:moments_assumption}, provided $\frac{\alpha}{\alpha-\eps}<1+\delta<2$. 
   
  \end{proof}
   
  Now we show that the identification of $\hat{D}$ indeeds identifies $L(r)=D(u_0,-\log r)$.

 \begin{lem}
  \label{lem:compare_D}
 Assume that \eqref{A1}-\eqref{A8} then  $\lim_{t \to \infty} \hat D(t)/D(u_0,t)=1$. In particular, ${\hat{D}(t+s)}/{\hat{D}(t)}$ converge to 1 as $t$ goes to infinity.
 \end{lem}
 \begin{proof}
  Recalling the definition of $h_t$ from Section \ref{sect:regular variation}, we have that
  \begin{align*}
   \hat D(t)/D(u_0,t)=\Ex[\eta]{\frac{D(U_{\s_1-1},t+S_{\s_1-1})}{D(u_0,t)}} ~=~ \Ex[\eta]{h_t(U_{\s_1-1},S_{\s_1-1})}
  \end{align*}
  
  Using Lemma \ref{lem:slowvar}, $\lim_{t \to \infty} h_t \equiv 1$. Lemmata  \ref{lem:exp_moment} and \ref{lem:bounds_on_LTa}  allow us to apply the dominated convergence theorem to obtain the assertion.
  \end{proof}

Now we can identify the slowly varying function.   
   
  \begin{thm}\label{thm:Llog}
   Assume that a function $\hat{D}$, such that ${\hat{D}(t+s)}/{\hat{D}(t)}\to1$ satisfies renewal equation \eqref{eq:renewal_hatD} with a directly Riemann integrable function $g$ and a nonarithmetic random variable $V_1$ such that $\Ex[\eta]{e^{\delta |V_1|}}<\8$ for some positive $\delta$. Then  $\lim_{t \to \infty} \hat D(t)/t $ exists and it is positive.
    \end{thm}
\begin{proof}[Source:]
 The proof is almost the same as the proof of Theorem 2.18 in \cite{DL1983}. Note that, although in \cite{DL1983} the derivative of $\hat D$ is used this can be easily avoided.
\end{proof}

\section{The Derivative Martingale}\label{sect:derivative martingale}

In this section, we finish the proof of Theorem \ref{thm:main2}, by proving the convergence of 
$$ \mathcal{W}_n(u) = \sum_{\abs{v}=n} \left[ S(v) { + } b(U(v)) \right] \, \est[\alpha](U(v)) e^{-\alpha S(v)} $$ to a nontrivial limit, which constitutes the exponent of fixed points. The assertions of Theorem \ref{thm:main2} are contained in the Theorem below, except for the identification of the slowly varying function, which was given in Section \ref{sect:L}, in particular in Theorem \ref{thm:Llog}.

%

\begin{thm}
Under Assumptions \eqref{A1}--{\eqref{A8}} and \eqref{A4c} or \eqref{A4f} instead of \eqref{A4}, the martingale $\mathcal{W}_n(u) $ for each $u \in \Sp$
has a nonnegative, nontrivial limit $\mathcal{W}(u)$, and $\LTfp(ru):= \Erw{e^{-r^\alpha \mathcal{W}(u)}}$ is a fixed point of $\ST$.
\end{thm}

\begin{proof}
Let $M(u)$ be the disintegration of the (up to scaling) unique fixed point of $\ST$ (described in Theorem \ref{thm:main}). By Theorem \ref{thm:Llog}, combined with Eq. \eqref{eq:slowvar} from Lemma \ref{lem:slowvar}, there is $K' \in (0,\infty)$ such that 
$$ \lim_{r \to 0} \sup_{u \in \Sp} \abs{\frac{1 -  \LTfp(r u) }{r^\alpha \est[\alpha](u) K'\abs{\log(r)}} - 1} ~=~0.$$
Then by \eqref{as2} from Proposition \ref{prop:dis},
$$ \lim_{n \to \infty} \, \sum_{\abs{v}=n} K' S^u(v) \est[\alpha](U^u(v)) e^{-\alpha S^u(v)} \frac{1 - \LTfp(e^{-S^u(v)} U^u(v))}{K' S^u(v) \est[\alpha](U^u(v)) e^{-\alpha S^u(v)}} ~=~ Z(u) \quad \Pfs $$
As a continuous function on $\Sp$, $u \mapsto b(u)$ is bounded, and by Lemma \ref{lem:maxpos}, 
$$ \lim_{n \to \infty} \sup_{\abs{v}=n} \abs{\frac{S^u(v) + b(U^u(v))}{S^u(v)}-1}=0.$$ Therefore, we can replace $S^u(v)$ by $S^u(v) + b(U^u(v))$, and obtain
$$ \lim_{n \to \infty} \,\sum_{\abs{v}=n} \big[S^u(v)+ b(U^u(v)) \big] \est[\alpha](U^u(v)) e^{-\alpha S^u(v)} ~=~ K' Z(u) \quad \Pfs$$
This shows the $\Pfs$ convergence of $\mathcal{W}_n(u)$ to $\mathcal{W}(u):=K'Z(u)$. Then $\P{\mathcal{W}(u)>0}=1$ by \eqref{as5} of Proposition \ref{prop:dis}. 
That $\LTfp(ru)= \Erw{e^{-r^\alpha \mathcal{W}(u)}}$ is a fixed point follows immediately, since $\Erw{e^{-r^\alpha K' Z(u)}}$ is a fixed point for any $K'>0$.
\end{proof}

%

\end{document}